\documentclass[reqno,centertags, 12pt]{amsart}
\usepackage{amsmath,amsthm,amscd,amssymb}
\usepackage{latexsym}
\usepackage{graphicx}
\usepackage[raiselinks,colorlinks]{hyperref}
\hypersetup{citecolor=blue}

\newcommand{\bbD}{{\mathbb{D}}}

\newcommand{\bbR}{{\mathbb{R}}}

\newcommand{\fre}{{\frak{e}}}

\newcommand{\calT}{{\mathcal T}}

\newcommand{\bdone}{{\boldsymbol{1}}}

\newcommand{\lb}{\label}
\newcommand{\f}{\frac}

\newcommand{\ti}{\tilde  }

\newcommand{\tr}{\text{\rm{Tr}}}

\newcommand{\ess}{\text{\rm{ess}}}
\newcommand{\ac}{\text{\rm{ac}}}

\newcommand{\supp}{\text{\rm{supp}}}

\newcommand{\bi}{\bibitem}

\newcommand{\beq}{\begin{equation}}
\newcommand{\eeq}{\end{equation}}
\newcommand{\ba}{\begin{align}}
\newcommand{\ea}{\end{align}}

\let\det=\undefined\DeclareMathOperator{\det}{det}

\newcounter{smalllist}

\allowdisplaybreaks
\numberwithin{equation}{section}

\newtheorem{theorem}{Theorem}[section]

\newtheorem*{p2.1}{Proposition 2.1}
\newtheorem{proposition}[theorem]{Proposition}
\newtheorem{lemma}[theorem]{Lemma}

\theoremstyle{definition}

\newtheorem{conjecture}[theorem]{Conjecture}

\theoremstyle{remark}
\newtheorem*{remark}{Remark}
\newtheorem*{remarks}{Remarks}
\newtheorem*{definition}{Definition}

\newcommand{\abs}[1]{\lvert#1\rvert}

\newcommand{\norm}[1]{\lVert#1\rVert}

\renewcommand{\MRhref}[2]{\href{http://www.ams.org/mathscinet-getitem?mr=#1}{#2}}
\renewcommand{\MR}[1]{}

\makeatletter
\def\@strippedMR{}
\def\@scanforMR#1#2#3\endscan{%
   \ifx#1M\ifx#2R\def\@strippedMR{#3}%
   \else\def\@strippedMR{#1#2#3}%
   \fi\fi}
\renewcommand\MR[1]{\relax\ifhmode\unskip\spacefactor3000 \space\fi
   \@scanforMR#1\endscan
   \MRhref{\@strippedMR}{MR\@strippedMR}}
\makeatother

\begin{document}
\title{Regularity and the Ces\`aro--Nevai Class}
\author{Barry Simon}

\thanks{Mathematics 253-37, California Institute of Technology, Pasadena, CA 91125.
E-mail: bsimon@caltech.edu. Supported in part by NSF grants DMS--0140592
and DMS-0652919}

\thanks{To be submitted to J.\ Approx.\ Theory}

\date{November 5, 2007}

\begin{abstract} We consider OPRL and OPUC with measures regular in the sense of
Ullman--Stahl--Totik and prove consequences on the Jacobi parameters or Verblunsky
coefficients. For example, regularity on $[-2,2]$ implies $\lim_{N\to\infty}
N^{-1} [\sum_{n=1}^N (a_n-1)^2 + b_n^2] =0$.
\end{abstract}

\maketitle

\section{Introduction and Background} \lb{s1}

This paper concerns the general theory of orthogonal polynomials on the real line, OPRL
(see \cite{Szb,Chi,FrB,Rice}), and the unit circle, OPUC (see \cite{Szb,GBk,OPUC1,OPUC2}).
Ullman \cite{Ull} introduced the notion of regular measure on $[-2,2]$ (he used
$[-1,1]$; we use the normalization more common in the spectral theory literature): a
measure, $d\mu$, on $\bbR$ with
\begin{equation} \lb{1.1}
\supp(d\mu) = [-2,2]
\end{equation}
and $(\{a_n,b_n\}_{n=1}^\infty$ are the Jacobi parameters of $d\mu$)
\begin{equation} \lb{1.2}
\lim_{n\to\infty} (a_1 \dots a_n)^{1/n} =1
\end{equation}

Here we will look at the larger class with \eqref{1.1} replaced by
\begin{equation} \lb{1.3}
\sigma_\ess (d\mu) = [-2,2]
\end{equation}
(i.e., $\supp(d\mu)$ is $[-2,2]$ plus a countable set whose only limit points are a subset
of $\{\pm 2\}$).

Our goal is to explore what restrictions regularity places on the Jacobi parameters. At first
sight, one might think \eqref{1.2} is the only restriction but, in fact, the combination of
both \eqref{1.2} and \eqref{1.3} is quite strong. This should not be unexpected. After all,
it is well known (going back at least to Nevai \cite{Nev79}; see also \cite[Sect.~13.3]{OPUC2})
that \eqref{1.1} plus $\liminf (a_1 \dots a_n) >0$ implies
\begin{equation} \lb{1.4}
\sum_{n=1}^\infty (a_n-1)^2 + b_n^2 <\infty
\end{equation}

One can use variational principles to deduce some restrictions on the $a$'s and $b$'s. For
example, picking $\varphi_n$ to be the vector in $\ell^2 (\{1,2,\dots \})$
\begin{equation} \lb{1.5}
\varphi_{n,j} = \begin{cases}
\f{1}{\sqrt{n}} & j\leq n \\
0 & j\geq n+1
\end{cases}
\end{equation}
and using the Jacobi matrix
\begin{equation} \lb{1.6}
J=
\begin{pmatrix}
b_1 & a_1 & 0 & 0 & \cdots \\
a_1 & b_2 & a_2 & 0 & \cdots \\
0 & a_2 & b_3 & a_3 & \cdots \\
\vdots & \vdots & \vdots & \vdots & \ddots
\end{pmatrix}
\end{equation}
one sees, for example, that \eqref{1.3} implies (see also Theorem~\ref{T1.2} below)
\begin{align}
b_n &\equiv 0 \Rightarrow \limsup_{n\to\infty} \f{1}{n} \sum_{j=1}^{n-1} a_j \leq 1  \lb{1.6a} \\
a_n &\equiv 1 \Rightarrow \lim_{n\to\infty} \f{1}{n} \sum_{j=1}^n b_j =0 \lb{1.6b}
\end{align}

In fact, we will prove much more:

\begin{theorem}\lb{T1.1} If $\mu$ obeys \eqref{1.3} and \eqref{1.2}, then
\begin{equation} \lb{1.7}
\lim_{n\to\infty} \f{1}{n} \sum_{j=1}^n (\abs{a_j-1} + \abs{b_j}) =0
\end{equation}
\end{theorem}

Following the terminology for the OPUC analog of this in Golinskii--Khrushchev \cite{KhGo}, we
call \eqref{1.7} the Ces\`aro--Nevai condition and $\{a_j,b_j\}_{j=1}^\infty$ obeying \eqref{1.7}
the Ces\`aro--Nevai class. It, of course, contains the Nevai class (named after \cite{Nev79})
where $\abs{a_j-1} + \abs{b_j}\to 0$.

Noting that $\supp(d\mu)$ bounded implies
\begin{equation} \lb{1.8}
A=\sup_n (\abs{a_n-1} + \abs{b_n}) <\infty
\end{equation}
and that, by the Schwarz inequality,
\begin{align}
\biggl( \f{1}{n}\sum_{j=1}^n \, \abs{a_j-1} + \abs{b_j}\biggr)^2
&\leq \f{2}{n} \sum_{j=1}^n (a_j-1)^2 + (b_j)^2 \notag \\
&\leq 2A\, \f{1}{n} \sum_{j=1}^n (\abs{a_j-1} + \abs{b_j}) \lb{1.9}
\end{align}
we see
\begin{equation} \lb{1.10}
\eqref{1.7} \Leftrightarrow \f{1}{n} \sum_{j=1}^n (a_j-1)^2 + (b_j)^2 \to 0
\end{equation}

While Theorem~\ref{T1.1} has a lot of information, it is not the whole story.
For example, if $a_n\equiv 1$, then by the same variational principle, for any
$j_k\to\infty$,
\[
\f{1}{n} \sum_{j_k}^{j_k+n} b_j\to 0
\]
It would be interesting to see what else can be said.

A major theme we explore is what can be said if $[-2,2]$ is replaced by a more general
set, $\fre$. In Section~\ref{s5}, we define Nevai and CN classes for finite gap sets $\fre$
and state a general conjecture which we prove in the special case where $d$ has $p$
components, each of harmonic measure $1/p$, that is, the periodic case with all gaps open.

In Section~\ref{s3}, we extend Theorem~\ref{T1.1} to the matrix OPRL case on $[-2,2]$,
and in Section~\ref{s6}, we use this and ideas of Damanik--Killip--Simon \cite{DKS2007}
to obtain the result in the last paragraph. Section~\ref{s4} has a brief discussion of
OPUC.

We should close by noting an earlier result of M\'at\'e--Nevai--Totik \cite{MNT87} 
related to---but neither stronger nor weaker than---Theorem~\ref{T1.2}: 

\begin{theorem}[\cite{MNT87}]\lb{T1.2} Suppose $\mu$ obeys \eqref{1.1} and $a_n\to 1$ as 
$n\to\infty$. Then $b_n\to 0$ as $n\to\infty$. 
\end{theorem} 

\begin{remarks} 1. $\mu$ need only obey \eqref{1.3} as seen by Remark~3 below. 

\smallskip 
2. This strengthens \eqref{1.6b}. There is no similar strengthening of \eqref{1.6a}. 

\smallskip 
3. One way of seeing this is as follows: By Last--Simon \cite{S304}, any right limit of 
a $J$ obeying \eqref{1.3} has $\sigma (J_r)\subset [-2,2]$ and has $a_n\equiv 1$. By a 
result of Killip--Simon \cite{KS} (see also \cite{S284,S285,S291}), any such $J_r$ has 
$b_n\equiv 0$. By compactness, $b_n\to 0$ for the original $J$. 
\end{remarks}

\medskip
It is a pleasure to thank Paul Nevai and Christian Remling for useful correspondence.

\section{OPRL on $[-2,2]$} \lb{s2}

Our goal here is to prove Theorem~\ref{T1.1}.

\begin{lemma}\lb{L2.1} Suppose $a_n\in (0,\infty)$ is a sequence so that
\begin{alignat}{2}
&\text{\rm{(i)}} \qquad && \liminf_{N\to\infty}\, (a_1 \dots a_N)^{1/N} \geq 1 \lb{2.1} \\
&\text{\rm{(ii)}} \qquad && \limsup_{N\to\infty}\, \f{1}{N}\sum_{n=1}^N a_n^2 \leq 1 \lb{2.2}
\end{alignat}
Then, as $N\to\infty$,
\begin{align}
& \f{1}{N} \sum_{n=1}^N a_n\to 1 \qquad\qquad \f{1}{N} \sum_{n=1}^N a_n^2 \to 1 \lb{2.3} \\
& \f{1}{N} \sum_{n=1}^N (a_n-1)^2 \to 0 \lb{2.4}
\end{align}
\end{lemma}

\begin{proof} By concavity of $\log x$ for all $x\in (0,\infty)$,
\[
\log x \leq x-1
\]
so \eqref{2.1} implies
\[
\liminf_{N\to\infty} \f{1}{N} \sum_{n=1}^N a_n \geq 1 + \liminf_{N\to\infty} \,
\log (a_1 \dots a_N)^{1/N} \geq 1
\]
Thus,
\[
\limsup \f{1}{N} \sum_{n=1}^N (a_n-1)^2 \leq 1-2+1 =0
\]
so \eqref{2.4} holds.

By the Schwarz inequality,
\[
\f{1}{N} \sum_{n=1}^N\, \abs{a_n-1} \leq \biggl[ \, \f{1}{N} \sum_{n=1}^N (a_n-1)^2\biggr]^{1/2}\to 0
\]
which implies the first limit in \eqref{2.3}. \eqref{2.4} and that limit imply \eqref{2.3}.
\end{proof}

\begin{proposition}\lb{P2.2} Let $\{a_n, b_n\}_{n=1}^\infty$ be the Jacobi parameters for a regular
measure with $\sigma_\ess (J)=[-2,2]$. Then
\begin{equation} \lb{2.5}
\f{1}{N} \biggl[ 2 \sum_{n=1}^{N-1} a_n^2 + \sum_{n=1}^N b_n^2 \biggr] \to 2
\end{equation}
as $N\to\infty$.
\end{proposition}

\begin{proof} Let $\{x_j^{(N)}\}_{j=1}^N$ be the zeros of the OPRL $p_N(x)$ associated to
the Jacobi parameters. Let $d\rho_{[-2,2]}$ be the equilibrium measures for $[-2,2]$
(see \cite{Land,Ran,EqMC} for potential theory notions). Since regularity implies that
the density of zeros converges to $d\rho_{[-2,2]}$ (see \cite{StT,EqMC}), we have
\begin{equation} \lb{2.6}
\f{1}{N} \sum_{n=1}^N (x_j^{(N)})^2 \to \int x^2 d\rho_{[-2,2]}(x)
\end{equation}
Since $\{x_j^{(N)}\}_{j=1}^N$ are the eigenvalues of the finite Jacobi matrix
\begin{equation} \lb{2.7}
J_{N;F} = \begin{pmatrix}
b_1 & a_1 \\
a_1 & b_2 & \ddots \\
{} & \ddots & \ddots & \ddots \\
{} & {} & \ddots & b_{N-1} & a_{N-1} \\
{} & {} & {} & a_{N-1} & b_N\,
\end{pmatrix}
\end{equation}
we have that
\begin{align}
\text{LHS of \eqref{2.6}} &= \f{1}{N} \tr (J_{N;F}^2) \notag \\
&= \f{1}{N} \biggl[\, \sum_{n=1}^N b_n^2 + 2 \sum_{n=1}^{N-1} a_n^2 \biggr] \lb{2.8}
\end{align}

Thus \eqref{2.5} is equivalent to
\begin{equation} \lb{2.9}
\int x^2 d\rho_{[-2,2]}(x)=2
\end{equation}

This can be seen either by using the explicit formula for $d\rho_{[-2,2]}$ (and $\int_0^\pi
(2\cos\theta)^2 \f{d\theta}{\pi} =2$) or by considering the special case $a_n\equiv 1$,
$b_n\equiv 0$ since the limit in \eqref{2.5} is the same for all regular $J$'s.
\end{proof}

\begin{proof}[Proof of Theorem~\ref{T1.1}] By regularity,
\begin{equation} \lb{2.10}
\liminf_{N\to\infty} \, (a_1 \dots a_N)^{1/N} =1
\end{equation}
and by Proposition~\ref{P2.2} and
\begin{equation} \lb{2.11}
\limsup\, \f{1}{N} \sum_{n=1}^{N-1} a_n^2 \leq 1
\end{equation}

By Lemma~\ref{L2.1}, we have \eqref{2.4}, and this and \eqref{2.5} imply
\begin{equation} \lb{2.12}
\f{1}{N} \sum_{n=1}^N b_n^2 \to 0
\end{equation}
By\eqref{1.10}, we get \eqref{1.7}.
\end{proof}

\section{MOPRL on $[-2,2]$} \lb{s3}

In this section, both for its own sake and because of the application in Section~\ref{s6}, we
want to consider matrix-valued measures for $[-2,2]$. Our reference for the associated OPRL
will be \cite{DPS2007} which discusses regular measures. $\ell$ is fixed and finite, and we
have a block Jacobi matrix of the form
\begin{equation}\lb{3.1}
J=
\begin{pmatrix}
B_1 & A_1 & 0 &  \cdots \\
A_1^\dagger & B_2 & A_2 &  \cdots \\
0 & A_2^\dagger & B_3 &  \cdots \\
\vdots & \vdots & \vdots &  \ddots
\end{pmatrix}
\end{equation}
where $A_j$ and $B_j$ are $\ell\times\ell$ matrices and $^\dagger$ is Hermitian conjugate. One
requires each $A_j$ is nonsingular.

Two sets of Jacobi parameters, $\{A_j,B_j\}_{j=1}^\infty$ and $\{\ti A_j, \ti B_j\}_{j=1}^\infty$,
are called equivalent if there exist $\ell\times\ell$ unitaries, $u_1 \equiv\bdone, u_2, u_3, \dots$
so that
\begin{equation}\lb{3.2}
\ti B_j = u_j^\dagger B_j u_j \qquad
\ti A_j = u_j^\dagger A_j u_{j+1}
\end{equation}
It is known (see \cite[Thm.~2.11]{DPS2007}) that there is a one-one correspondence between
nontrivial $\ell\times\ell$ matrix-valued measures, $d\mu$, (with nontriviality suitably
defined) and equivalence classes of Jacobi parameters.

$\{A_j,B_j\}_{j=1}^\infty$ is called type~1 (resp.\ type~3) if each $A_j$ is positive (resp.\
$A_j$ is lower triangular and positive on diagonal). Moreover (\cite[Thm.~2.8]{DPS2007}), each
equivalence class has exactly one representative of type~1 and one of type~3. An $\ell\times
\ell$ matrix-valued measure is called regular (\cite[Ch.~5]{DPS2007}) for $[-2,2]$ if and only if
\begin{equation}\lb{3.3}
\sigma_\ess (d\mu) = [-2,2]
\end{equation}
and
\begin{equation}\lb{3.4}
\biggl[\, \prod_{n=1}^N \, \abs{\det(A_n)}\biggr]^{1/N} \to 1
\end{equation}
Our basic result for such MOPRL is:

\begin{theorem}\lb{T3.1} If $\{A_n,B_n\}_{n=1}^\infty$ are the Jacobi parameters for an $\ell\times
\ell$ matrix-valued measure which is regular for $[-2,2]$ and are either of type~1 or type~3, then
\begin{equation}\lb{3.5}
\f{1}{N}\sum_{n=1}^N \, \norm{A_n-\bdone} + \norm{B_n} \to 1
\end{equation}
\end{theorem}

\begin{remark} \eqref{3.5} does not hold for all equivalent $\{\ti A_n,\ti B_n\}_{n=1}^\infty$,
but it is easy to see that
\begin{equation}\lb{3.6}
\f{1}{N} \sum_{n=1}^N \, \norm{A_n^* A_n-\bdone} + \norm{B_n} \to 0
\end{equation}
is equivalence class independent and implied by \eqref{3.5} for the type~1 or type~3
representative.
\end{remark}

\begin{proof} We consider type~3 first. By Thm.~5.2 of \cite{DPS2007}, the density of zeros
converges to the equilibrium measure, so analogously to \eqref{2.5},
\begin{equation}\lb{3.7}
\f{1}{N\ell} \biggl[ 2 \sum_{n=1}^{N-1} \tr (A_n^* A_n) + \sum_{n=1}^N
\tr (B_n^* B_n)\biggr] \to 2
\end{equation}

In the type~3 case, \eqref{3.4} says
\begin{equation}\lb{3.8}
\biggl[\, \prod_{n=1}^N \, \prod_{j=1}^\ell (A_n)_{jj}\biggr]^{1/N\ell} \to 1
\end{equation}
so as in the proof of Theorem~\ref{T1.1}, we find
\begin{equation}\lb{3.9}
\f{1}{N\ell} \sum_{n=1}^N\, \sum_{j=1}^\ell\, \abs{(A_n)_{jj}-1}^2 \to 0
\end{equation}
and then that
\begin{equation}\lb{3.10}
\f{1}{N} \sum_{n=1}^N \tr (B_n^* B_n) \to 0
\end{equation}
and
\begin{equation}\lb{3.11}
\f{1}{N} \sum_{n=1}^N \abs{\tr (A_n^* A-\bdone)} \to 0
\end{equation}

In the type~1 case, one uses the inequality
\begin{equation}\lb{3.12}
A\geq 0 \Rightarrow \det(A) \leq \prod_{j=1}^\ell A_{jj}
\end{equation}
(see Simon \cite[Cor.~8.10]{STI}) and the fact that Lemma~\ref{L2.1}
only requires an inequality in \eqref{2.1}.
\end{proof}

\section{OPUC} \lb{s4}

Here we will prove two results about OPUC. Recall $d\mu$ on $\partial\bbD$ with
$\sigma_\ess(d\mu)=\fre$ is called regular if and only if
\begin{equation} \lb{4.1}
\lim_{N\to\infty} \biggl(\, \prod_{j=0}^{N-1}\rho_j\biggr)^{1/N} = C(\fre)
\end{equation}
the capacity of $\fre$ where $\rho_j = (1-\abs{\alpha_j}^2)^{1/2}$ and
$\{\alpha_j\}_{j=0}^\infty$ are the Verblunsky coefficients.

\begin{theorem}\lb{T4.1} Let $d\mu$ be a measure of $\partial\bbD$ regular for $\fre=
\partial\bbD$. Then, as $N\to\infty$,
\begin{equation} \lb{4.2}
\f{1}{N} \sum_{j=0}^{N-1}\, \abs{\alpha_j} \to 0
\end{equation}
\end{theorem}

\begin{remark} This is the original CN class of \cite{KhGo}.
\end{remark}

\begin{proof} $C(\partial\bbD)=1$, so by Lemma~\ref{L2.1} and
\begin{equation} \lb{4.3}
\f{1}{N} \sum_{j=0}^{N-1} \rho_j^2 \leq 1
\end{equation}
we obtain
\begin{equation} \lb{4.4}
\f{1}{N} \sum_{j=0}^{N-1} (1-\rho_j^2)\to 0
\end{equation}
which implies \eqref{4.2} by the Schwarz inequality.
\end{proof}

For $a\in (0,1)$, let $\Gamma_a$ be the arc
\begin{equation} \lb{4.5}
\{z\in\partial\bbD\mid z=e^{i\theta},\, \pi\geq\abs{\theta} > 2\arcsin(a)\}
\end{equation}
which has capacity $a$. Then

\begin{theorem}\lb{T4.2} Let $d\mu$ be a measure on $\partial\bbD$, regular for $\fre=
\Gamma_a$. Then as $N\to\infty$,
\begin{alignat}{2}
& \text{\rm{(a)}} \qquad && \f{1}{N} \sum_{j=0}^{N-1}\, (\abs{\alpha_j}-a)^2 \to 0 \lb{4.6} \\
& \text{\rm{(b)}} \qquad && \f{1}{N} \sum_{j=0}^{N-1}\, \abs{\alpha_{j+1} -\alpha_j}^2 \to 0 \lb{4.7} \\
\intertext{For any $k$,}
& \text{\rm{(c)}} \qquad && \f{1}{N} \sum_{j=0}^{N-1} \min_\theta \biggl(\, \sum_{\ell=1}^k\,
\abs{\alpha_{j+\ell} -ae^{i\theta}}^2\biggr) \to 0 \lb{4.8}
\end{alignat}
\end{theorem}

\begin{remark} The isospectral torus for $\Gamma_a$ is exactly $\{\{\alpha_j\equiv ae^{i\theta}\}
\}_{\theta\in [0,2\pi)}$, that is, the constant sequence of Verblunsky coefficients, so (c) involves
an approach to an isospectral torus.
\end{remark}

\begin{proof} By regularity and the connection between zeros of POPUC and eigenvalues of
finite CMV matrices as defined in \cite{CMVrev}, one has that
\begin{equation} \lb{4.9}
\f{1}{N} \sum_{n=0}^{N-1} -\bar\alpha_{n+1} \alpha_n \to c
\end{equation}
where $c$ is the first moment of the equilibrium measure, that is, $\int z\, d\rho_{\Gamma_a}(z)$.
Specializing to the case $\alpha_n \equiv a$ to evaluate $c$, we see that
\begin{equation} \lb{4.10}
\f{1}{N} \sum_{n=0}^{N-1} \bar\alpha_{n+1} \alpha_n \to a^2
\end{equation}

On the other hand, by regularity,
\begin{equation} \lb{4.10a}
\f{1}{N} \sum_{n=0}^{N-1} \log (1-\abs{\alpha_n}^2) \to \log (1-\abs{a}^2)
\end{equation}
and by concavity of $\log$,
\[
\log (1-x) - \log (1-\abs{a}^2) \leq \f{1}{1-\abs{a}^2}\, (\abs{a}^2 -x)
\]
so
\begin{equation} \lb{4.11}
\liminf \f{1}{N} \sum_{n=0}^{N-1} \, (\abs{a}^2 - \abs{\alpha_n}^2) \geq 0
\end{equation}
and thus,
\begin{equation} \lb{4.12}
\limsup \f{1}{N} \sum_{n=0}^{N-1}\, \abs{\alpha_n}^2 \leq a^2
\end{equation}

By \eqref{4.10} and the Schwarz inequality,
\begin{equation} \lb{4.13}
\liminf \f{1}{N} \sum_{n=0}^{N-1} \, \abs{\alpha_n}^2 \geq a^2
\end{equation}
so
\begin{equation} \lb{4.14}
\f{1}{N} \sum_{n=0}^{N-1}\, \abs{\alpha_n}^2 \to a^2
\end{equation}

For $y\in (0,1]$ (by Taylor's theorem with remainder and $\max_{(0,1]} \f{d^2}{dy^2}
\log(y)=-1$),
\[
\log(y) -\log(1-\abs{a}^2) - \biggl[ \f{y-(1-\abs{a}^2)}{1-\abs{a}^2}\biggr]
\leq -\f12 \, (y-(1-\abs{a}^2))^2
\]
so \eqref{4.10a} and \eqref{4.14} imply
\[
\f{1}{N} \sum_{n=0}^{N-1}\, \abs{\abs{\alpha}^2 -a^2} \to 0
\]
which implies \eqref{4.7}.

\eqref{4.14} and \eqref{4.10} imply \eqref{4.7}. Finally, \eqref{4.6} and \eqref{4.7}
imply \eqref{4.8}.
\end{proof}

\section{The Nevai and CN Classes} \lb{s5}

In \cite{OPUC2}, I proposed using approach to an isospectral torus as a replacement for the Nevai
class when $[-2,2]$ is replaced by the spectrum of a periodic Jacobi matrix. This idea was then
implemented in Last--Simon \cite{S304} and Damanik--Killip--Simon \cite{DKS2007}. The latter
discussed extending this notion to a general finite gap set, and this idea was further developed
in Remling \cite{Remppt}.

$\fre$ will denote a finite gap set, that is,
\begin{equation} \lb{5.1a}
\fre=[\alpha_1,\beta_1] \cup [\alpha_2,\beta_2]\cup\cdots\cup
[\alpha_{\ell+1}, \beta_{\ell+1}] \subset\bbR
\end{equation}
where
\begin{equation} \lb{5.1b}
\alpha_1 < \beta_1 < \alpha_2 < \beta_2 < \cdots < \beta_{\ell+1}
\end{equation}
Given such a set, there is a natural torus, $\calT_\fre$, of almost periodic Jacobi matrices,
discussed, for example, in \cite{SY,CSZ1}; it can be described \cite{Remppt} as the
restriction to $\{1, \dots\}$ of the two-sided reflectionless Jacobi matrices, $J^\sharp$,
with
\begin{equation} \lb{5.1c}
\sigma (J^\sharp) =\fre
\end{equation}
All $J\in\calT_\fre$ have
\begin{equation} \lb{5.1d}
\sigma_\ess (J)=\fre
\end{equation}
$\calT_\fre$ is a torus in the uniform topology as well as the product topology.

Given a pair of bounded Jacobi parameters, $J=\{a_n,b_n\}_{n=1}^\infty$, $\ti J=
\{\ti a_n,\ti b_n\}_{n=1}^\infty$, define $d_m (J,\ti J)$ by
\begin{equation} \lb{5.1}
d_m (J,\ti J)=\sum_{k=0}^\infty e^{-\abs{k}} (\abs{a_{m+k} -\ti a_{m+k}} +
\abs{b_{m+k} -\ti b_{m+k}})
\end{equation}
If $\calT_\fre$ is an isospectral torus, let
\begin{equation} \lb{5.2}
d_m (J,\calT_\fre) = \inf_{\ti J \in \calT_\fre} d_m (J,\ti J)
\end{equation}

\begin{definition} If $\fre\subset\bbR$ is a finite gap set, we define the Nevai class
$N(\fre)$ to be those $J$'s with
\begin{equation} \lb{5.3}
\lim_{m\to\infty} d_m (J,\calT_\fre)=0
\end{equation}
This is equivalent (by compactness) to saying all the right limits of $J$ lie in $\calT_\fre$.
\end{definition}

It is a theorem of Last--Simon \cite{S304} that
\begin{equation} \lb{5.4}
J\in N(\fre)\Rightarrow \sigma_\ess (J)=\fre
\end{equation}
and of Remling \cite{Remppt} that
\begin{equation} \lb{5.5}
\sigma_\ess (J) = \sigma_\ac (J) = \fre \Rightarrow J\in N(\fre)
\end{equation}

It is not hard to see that
\[
J\in N(\fre) \Rightarrow J \text{ is regular for $\fre$}
\]

Analogously, we define the Ces\`aro--Nevai class, $CN(\fre)$, as those $J$ with
\begin{equation} \lb{5.6}
\f{1}{N} \sum_{m=1}^N d_m (J,\calT_\fre) \to 0
\end{equation}

A main conjecture we make in this note is:

\begin{conjecture}\lb{Con5.1} If $J$ is regular for $\fre$, that is, $\sigma_\ess(J)=\fre$,
and $(a_1 \dots a_N)^{1/N}\to C(\fre)$, then $J\in CN(\fre)$.
\end{conjecture}

In the next section, we will prove this for a special class of $\fre$'s. Of course, we make a
similar conjecture for finite gap OPUC. Indeed, Theorem~\ref{T4.2} is the case of OPUC with
one gap!

\section{Generic Periodic Spectrum} \lb{s6}

Our goal is to prove:

\begin{theorem}\lb{T6.1} Let $\fre$ be a finite gap set so that each $[\alpha_j,\beta_j]$ has harmonic
measure $(\ell+1)^{-1}$ {\rm{(}}equivalently, there is a $J_0$ with period $\ell+1$ so $\fre=\sigma_\ess
(J_0)${\rm{)}}. Let $J$ be a Jacobi matrix with regular spectral measure so that $\sigma_\ess (J)=\fre$.
Then $J\in CN(\fre)$.
\end{theorem}

We use $p$ for $\ell+1$, the period of $J_0$.

Following \cite{DKS2007}, we exploit $\Delta_{J_0}(J)$ where $\Delta_{J_0}$ is the discriminant
\cite{DKS2007,Rice} of $J_0$, a polynomial of degree $p$. If $J$ is any Jacobi matrix, $\Delta_{J_0}(J)$
is a $p\times p$ block Jacobi matrix of type~3. We use $A_{J_0,k}(J)$ and $B_{J_0,k}(J)$ to denote
the $p\times p$ matrix blocks in $\Delta(J)$.

\cite{DKS2007} proved the following theorem (their Thm.~11.12); here $\norm{\cdot}$ is the
Hilbert--Schmidt norm.

\begin{theorem}[\cite{DKS2007}]\lb{T6.2} Fix $J_0$ periodic with $\sigma_\ess (J_0)=\fre$ and
$J$ an arbitrary bounded Jacobi matrix. Then
\begin{equation} \lb{6.1}
\sum_{k=1}^\infty\, \norm{A_{J_0,k}(J)-\bdone}_2^2 + \norm{B_{J_0,k}(J)}_2^2 <\infty
\end{equation}
if and only if
\begin{equation} \lb{6.2}
\sum_{k=1}^\infty d_k (J,\calT_\fre)^2 <\infty
\end{equation}
\end{theorem}

Because this comparison is local, the exact same proof shows

\begin{theorem}\lb{T6.3} Let $J_0$ be periodic with $\sigma_\ess (J_0)=\fre$ and $J$ an arbitrary
Jacobi matrix. Then
\begin{equation} \lb{6.3}
\lim_{N\to\infty} \f{1}{N} \sum_{k=1}^N\, [\norm{A_{J_0,k}(J)-\bdone}_2^2 +
\norm{B_{J_0,k}(J)}^2] =0
\end{equation}
if and only if
\begin{equation} \lb{6.4}
\lim_{N\to\infty} \f{1}{N} \sum_{k=1}^N d_k (J,\calT_\fre)^2 \to 0
\end{equation}
\end{theorem}

With this and Theorem~\ref{T3.1}, we can prove Theorem~\ref{T6.1}.

\begin{proof}[Proof of Theorem~\ref{T6.1}] $\Delta(x)$ has the form
\[
\Delta(x) = (a_{0,1}\, a_{0,2} \dots a_{0,p})^{-1} x^p + \text{lower order}
\]
so the diagonal matrix elements of $\Delta(J)$ are
\[
\f{a_j a_{j+1} \dots a_{j+p}}{a_{0,j} \dots a_{0,j+p}} \equiv \alpha_{jj}
\]

If $J$ is regular, for $\fre$,
\begin{equation} \lb{6.5}
\biggl[ \f{a_1\dots a_n}{C(\fre)^n}\biggr]^{1/n} \to 1
\end{equation}
But $a_{0,j}\dots a_{0,j+p} = C(\fre)^p$ for periodic Jacobi matrices, so \eqref{6.5} implies
\[
(\alpha_{11} \alpha_{22} \dots \alpha_{nn})^{1/n} \to 1
\]
which implies that $\Delta(J)$ is a regular block Jacobi matrix.

By Theorem~\ref{T3.1}, \eqref{6.3} holds and so, by Theorem~\ref{T6.3}, we have the $CN(\fre)$
condition \eqref{6.4}.
\end{proof}

\bigskip

\end{document}